\documentclass[12pt]{amsart}

\usepackage{extsizes,comment,amsmath,amssymb}
\usepackage[margin=1in]{geometry}
\usepackage{graphicx}
\usepackage{pgfplots,mathtools,float,mathrsfs}
\usepackage{booktabs,enumitem}
\usepackage{pgfplotstable}
\usepgfplotslibrary{groupplots}
\pgfplotsset{compat=newest}

\usepackage{xcolor}
\definecolor{colorLtwo}{HTML}{1F77B4}
\definecolor{colorHone}{HTML}{2CA02C}
\definecolor{colorHtwo}{HTML}{D62728}
\definecolor{colorBell}{HTML}{1F77B4}
\definecolor{colorArg}{HTML}{D62728}
\definecolor{colorSpec}{HTML}{2CA02C}

\usepackage[colorlinks=true,linkcolor=blue,citecolor=green!65!black,urlcolor=blue]{hyperref}
\usepackage[sort,compress]{cite}
\usepackage[capitalise,nameinlink]{cleveref}

\numberwithin{equation}{section}

\newcommand{\dd}{\,\mathrm d}
\newcommand{\per}{\mathrm{per}}

\newcommand{\citep}{\cite}

\theoremstyle{plain}
\newtheorem{theorem}{Theorem}[section]
\newtheorem{lemma}[theorem]{Lemma}
\newtheorem{proposition}[theorem]{Proposition}

\theoremstyle{definition}

\theoremstyle{remark}
\newtheorem{remark}[theorem]{Remark}

\begin{document}
\footnotetext[1]{Institute of Mathematics, Johannes Gutenberg University, Mainz, Germany}
\footnotetext[2]{Faculty of Mathematics, University of Vienna, Vienna, Austria}
\author[A.~Brunk \& M.~Fritz]{Aaron Brunk$^{1}$ \and Marvin Fritz$^{2}$}
\begin{center}
      \bf \MakeUppercase{High-order conforming finite elements for the Cahn--Hilliard equation:
Relative-energy stability and energy defects} 
\end{center}

\begin{center}
   \sc Aaron Brunk$^{1}$, Marvin Fritz$^{2}$
\end{center}

\begin{quote}\small
\textsc{Abstract.}
We study a semidiscrete single-field Galerkin approximation of the
Cahn--Hilliard equation using high-order conforming finite element spaces.
More specifically, globally \(C^1\) finite elements with \(H^2\)-conforming
trial spaces, including Argyris, Bell, and Bogner--Fox--Schmit elements,
allow a direct discretization of the fourth-order formulation and preserve
mass exactly.
The main structural result is an exact energy balance for the physical
Cahn--Hilliard energy. Besides the expected discrete dissipation,
the balance contains an explicitly computable energy defect. This defect
vanishes for Laplacian-invariant periodic spaces, such as Fourier spaces,
but is generally nonzero for classical \(C^1\) finite elements. It therefore
quantifies the precise loss of a discrete gradient-flow structure.
We prove semidiscrete a priori error estimates by a relative-energy
argument. The estimate is closed using an augmented relative energy and a
discrete elliptic reconstruction bound for the inverse discrete Laplacian.
The resulting convergence rates match the expected approximation orders. Numerical experiments with Bell and Argyris
elements confirm the rates and demonstrate the defect mechanism by
comparison with a Fourier reference discretization. \smallskip

 \smallskip

\end{quote}

\section{Introduction}\label{sec:intro}

The Cahn--Hilliard equation was originally proposed to describe phase
separation in binary alloys~\cite{cahn1958free} and has since become a
canonical model for diffuse-interface dynamics, including spinodal
decomposition, tumour growth~\cite{fritz2023tumor}, and microstructure
evolution. A key structural feature of the model is its gradient-flow
interpretation in an $H^{-1}$ metric, which implies mass conservation and
dissipation of the Ginzburg--Landau free energy at the continuous level.

Most numerical methods for the Cahn--Hilliard equation use mixed
formulations that introduce the chemical potential as an additional unknown;
see, for example,~\cite{barrett1999finite,barrett2001fully,brunk2023stability,brunk2025review} and the references therein. Mixed methods are flexible and
admit many energy-stable time discretizations. However, they also increase
the number of unknowns and obscure, to some extent, the direct relation
between the phase field and the fourth-order operator. 

Motivated by classical globally $C^1$ finite elements,
we revisit the single-field fourth-order formulation. We consider conforming
trial spaces
$
V_h\subset H^2_{\text{per}}(\Omega),
$
built from classical $C^1$ elements such as the Argyris element
\cite{argyris-0}, the Bell element~\cite{bell-0}, and the
Bogner--Fox--Schmit element~\cite{bogner1965generation}. Their approximation
properties for fourth-order problems are classical; see, for instance,
\cite{ciarlet2002finite,brenner2008mathematical}. Such spaces make it
possible to approximate the fourth-order single-field problem directly,
without auxiliary variables and without stabilization parameters of the type
often needed for $C^0$ discretizations of fourth-order equations.

High-regularity discretizations for the Cahn--Hilliard equation and related
phase-field models have been pursued from several directions. For conforming
$C^1$ finite element methods we refer to~\cite{stogner2008c1,feng2008apost,egger2025feedback};
spline- and NURBS-based approaches through isogeometric analysis were
introduced in~\cite{gomez2008isogeometric,anders2011diffusion}; and
$H^2$-conforming virtual element methods were developed more recently in
\cite{antonietti2016c1virtual,adak2024h2vem,leng2025hofrac}. Nonconforming
Morley-type methods have also been analysed as computationally cheaper
alternatives~\cite{li2019morley,wu2020morley}. The focus of the present
paper is complementary: we analyse the exact relation between the
single-field $H^2$-conforming Galerkin method and the Cahn--Hilliard
gradient-flow structure.

The first contribution is an exact energy identity for the physical
Cahn--Hilliard energy restricted to the finite element space. If
$\phi_h$ denotes the semidiscrete solution, then
$$
\frac{d}{dt}E_h(\phi_h(t))
=
-\|\partial_t\phi_h(t)\|_{H^{-1}_h}^2+\mathcal R_h(t),
$$
where $\|\cdot\|_{H^{-1}_h}$ is the natural discrete $H^{-1}$ norm and
$\mathcal R_h$ is an explicitly computable energy defect. This defect
vanishes when the trial space is invariant under the Laplacian, as in
Fourier or trigonometric spectral discretizations on periodic domains. In
that case, the semidiscrete method is exactly energy dissipative. For
classical globally $C^1$ finite element spaces, however, the defect is
generally nonzero and measures the deviation from an exact discrete
gradient-flow structure for the physical energy.

The second contribution is a relative-energy stability framework.
We compare the semidiscrete solution $\phi_h$ with a perturbed semidiscrete
trajectory $\widehat\phi_h$ and derive a relative-energy identity. The
corresponding relative defect is the analogue of the physical defect in the
two-trajectory setting. To control it, we introduce an augmented relative
energy, adding a small $H^2$-level correction. This augmentation yields a
strong $L^2$-in-time control of $\partial_t(\phi_h-\widehat\phi_h)$, which is
precisely what is needed to absorb the relative defect. The key technical
ingredient is a discrete elliptic-reconstruction estimate for
$(-\Delta_h)^{-1}$, proved by viewing the inverse discrete Laplacian as a
Poisson Ritz projection.

Finally, we apply the abstract relative-energy stability theorem to the exact
solution by taking $\widehat\phi_h$ to be a mean-preserving biharmonic Ritz
projection of $\phi$. A standard $L^2$-relative estimate yields the optimal
$L^\infty(0,T;L^2)$ rate, while the augmented relative-energy estimate gives
the necessary $H^2$ control. Combining the two by interpolation gives the
optimal $L^\infty(0,T;H^1)$ rate.

The paper is organized as follows. In \Cref{sec:prelim} we introduce the
the $H^2$-conforming spaces, the discrete Laplacian, and
the biharmonic Ritz projection. In \Cref{sec:continuous} we recall the
continuous gradient-flow structure. The semidiscrete scheme and the physical
energy defect are analysed in \Cref{sec:sd-defect}. The relative-energy
framework is developed in \Cref{sec:relative}. In \Cref{sec:error} we apply
the framework to the exact solution and derive semidiscrete error estimates.
In \Cref{sec:numerics} we present numerical experiments for Bell and Argyris
elements that confirm the predicted convergence rates and contrast the
energy defect of classical $C^1$ finite element spaces with that of a
Fourier reference, for which the defect vanishes up to roundoff. The key
technical tool used throughout the analysis, an $H^2$-stable elliptic
reconstruction bound for the inverse discrete Laplacian, is collected in
\Cref{app:reconstruction}.

\section{Preliminaries}\label{sec:prelim}

Throughout, all constants denoted by $C$ may change from line to line but are
independent of the mesh size $h$.

\subsection{Setting and notation}
Let $\Omega\subset\mathbb R^d$, $d \in \{2,3\}$, be a periodic box, identified with the flat
$d$-dimensional torus. Throughout we work in the periodic setting; all
functions and derivatives are understood to be periodic across opposite faces
of~$\Omega$.


We use $(\cdot,\cdot)$ for the $L^2(\Omega)$ inner product and
$\|\cdot\|$ for the induced norm. For $s\ge0$, $\|\cdot\|_{H^s}$ denotes the
usual Sobolev norm on $H^s_{\per}(\Omega)$. We define the mean-free spaces
$$
\begin{aligned}
\mathring L^2(\Omega)
&:=
\{v\in L^2(\Omega):(v,1)=0\}, \\
\mathring H^s_{\per}(\Omega)
&:=
\{v\in H^s_{\per}(\Omega):(v,1)=0\}.
\end{aligned}
$$
On mean-free periodic functions we use the standard equivalences
\begin{align}\label{eq:poincare-H1}
\|v\|_{H^1}\simeq \|\nabla v\|,
\qquad v\in \mathring H^1_{\per}(\Omega),
\\ \label{eq:poincare-H2}
\|v\|_{H^2}\simeq \|\Delta v\|,
\qquad v\in \mathring H^2_{\per}(\Omega).
\end{align}


\subsection{$H^2$-conforming finite element spaces}

Let $\mathcal T_h$ be a quasi-uniform, shape-regular triangulation or
quadrangulation of $\Omega$ compatible with periodicity. We consider a finite
element space
$$
V_h\subset C^1(\overline\Omega)\cap H^2_{\per}(\Omega)
$$
built from classical globally $C^1$ elements. Examples include:
\begin{itemize}[leftmargin=1.8em] \itemsep0em
  \item the Argyris element, see \cite{argyris-0} and
        \cite[Thm.~2.2.11]{ciarlet2002finite};
  \item the Bell element, see \cite{bell-0} and
        \cite[Thm.~2.2.12]{ciarlet2002finite};
  \item the Bogner--Fox--Schmit element, see
        \cite{bogner1965generation} and
        \cite[Thm.~2.2.14]{ciarlet2002finite}.
\end{itemize}
We set the mean-free finite element space
$$
\mathring V_h:=\{v_h\in V_h:(v_h,1)=0\},
$$
and denote by $\Pi_h^0:L^2(\Omega)\to\mathring V_h$ the $L^2$-orthogonal
projection onto $\mathring V_h$.
Further, the approximation order is denoted by $p\ge2$; in particular, for the above element
families, one has
$$
p=
\begin{cases}
5, & \text{Argyris},\\
4, & \text{Bell},\\
3, & \text{Bogner--Fox--Schmit}.
\end{cases}
$$

We use the standard inverse inequality
\begin{equation}\label{eq:inverse-H2-H1}
\|v_h\|_{H^2}\le C h^{-1}\|v_h\|_{H^1}
\qquad\forall v_h\in V_h.
\end{equation}
Further, we use the standard quasi-interpolation
operator
$\mathcal I_h\colon H^1_{\per}(\Omega)\to V_h$,
which is obtained by combining the local nodal interpolation associated with the
underlying $C^1$ element with averaging over local mesh patches, in the
spirit of the Scott--Zhang and Cl\'ement constructions: pointwise values of
derivatives, which are not well defined on $H^1$, are replaced by suitable
local averages, while $\mathcal I_h$ reproduces the polynomial space of
$V_h$ on each patch. For the relevant local approximation theory of the
underlying $C^1$ elements we refer to
\cite[Sec.~6.1]{brenner2008mathematical} and
\cite[Thm.~3.5.1]{ciarlet2002finite}; the $H^1$-stable Hermite variant
applicable to Argyris-, Bell-, and Bogner--Fox--Schmit-type elements is
constructed in \cite{girault2002hermite}. The resulting operator satisfies
the following bounds.

\begin{lemma}[Quasi-interpolant]\label{lem:interpolant}
The operator $\mathcal I_h$ satisfies, for all $\psi\in H^1_{\per}(\Omega)$,
\begin{align}
\|\mathcal I_h\psi\|_{H^1}
&\le C\|\psi\|_{H^1}, \label{eq:I-H1-stab}\\
\|\psi-\mathcal I_h\psi\|
&\le C h\|\psi\|_{H^1}, \label{eq:I-L2-H1}\\
\|\Delta\mathcal I_h\psi\|
&\le C h^{-1}\|\psi\|_{H^1}, \label{eq:I-lap-H1}
\end{align}
and, for all $v\in H^2_{\per}(\Omega)$,
\begin{align}
\|v-\mathcal I_h v\|_{H^1}
&\le C h\|v\|_{H^2}, \label{eq:I-H1-H2}\\
\|\mathcal I_h v\|_{H^2}
&\le C\|v\|_{H^2}. \label{eq:I-H2-stab}
\end{align}
\end{lemma}



\subsection{Discrete Laplacian and discrete negative norms}

We define the discrete Laplacian $\Delta_h:V_h\to V_h$ by
\begin{equation}\label{eq:Delta-h}
(\Delta_h u_h,v_h)=-(\nabla u_h,\nabla v_h)
\qquad\forall u_h,v_h\in V_h.
\end{equation}
Then $\Delta_h$ is self-adjoint in $L^2(\Omega)$ and maps
$\mathring V_h$ into $\mathring V_h$. Moreover, $-\Delta_h$ is positive
definite on $\mathring V_h$.
Further, for $g_h\in\mathring V_h$, we define its inverse
$(-\Delta_h)^{-1}g_h\in\mathring V_h$ by
\begin{equation}\label{eq:discrete-inverse}
(\nabla(-\Delta_h)^{-1}g_h,\nabla v_h)=(g_h,v_h)
\qquad\forall v_h\in\mathring V_h.
\end{equation}
The corresponding discrete $H^{-1}$ norm is then given by
\begin{equation}\label{eq:Hminus1h}
\|g_h\|_{H^{-1}_h}^2
:=
(g_h,(-\Delta_h)^{-1}g_h)
=
\|\nabla(-\Delta_h)^{-1}g_h\|^2.
\end{equation}
In the same manner, we also use the discrete $H^{-2}$ dual norm
\begin{equation}\label{eq:Hminus2h}
\|\rho_h\|_{H^{-2}_h}
:=
\sup_{0\ne v_h\in\mathring V_h}
\frac{(\rho_h,v_h)}{\|v_h\|_{H^2}},
\qquad \rho_h\in\mathring V_h.
\end{equation}

\subsection{Biharmonic Ritz projection}

For $v\in H^2_{\per}(\Omega)$ we define the biharmonic Ritz projection
$R_hv\in V_h$ by
\begin{equation}\label{eq:ritz}
\bigl(\Delta(v-R_hv),\Delta\psi_h\bigr)+(v-R_hv,\psi_h)=0
\qquad\forall \psi_h\in V_h.
\end{equation}


\begin{lemma}[Well-posedness and mass preservation]\label{lem:ritz-wp}
For every $v\in H^2_{\per}(\Omega)$ there exists a unique $R_hv\in V_h$
satisfying \eqref{eq:ritz}. Moreover, $R_h$ preserves the mean,
\begin{equation}\label{eq:ritz-mean}
\overline{R_hv}=\overline v.
\end{equation}
\end{lemma}

\begin{proof}
The bilinear form
$$
a(u,w):=(\Delta u,\Delta w)+(u,w)
$$
is symmetric, continuous on $H^2_{\per}(\Omega)$, and coercive there since
$\|u\|_{H^2}^2\simeq\|\Delta u\|^2+\|u\|^2$ on the torus. Its restriction to
$V_h$ is therefore coercive as well, and existence and uniqueness follow
from the Lax--Milgram lemma. To see \eqref{eq:ritz-mean}, take
$\psi_h\equiv1\in V_h$ in \eqref{eq:ritz}: since $\Delta1=0$, this yields
$(v-R_hv,1)=0$, i.e.\ $\overline{R_hv}=\overline v$.
\end{proof}

\begin{lemma}[Ritz approximation]\label{lem:ritz-error}
Let $v\in H^{p+1}(\Omega)$. Then
\begin{align}
\|v-R_hv\|_{H^2}
&\le C h^{p-1}\|v\|_{H^{p+1}}, \label{eq:ritz-H2}\\
\|v-R_hv\|_{H^1}
&\le C h^{p}\|v\|_{H^{p+1}}, \label{eq:ritz-H1}\\
\|v-R_hv\|
&\le C h^{p+1}\|v\|_{H^{p+1}}, \label{eq:ritz-L2}\\
\|v-R_hv\|_{H^{-1}}
&\le C h^{\min(p+2,\,2p-2)}\|v\|_{H^{p+1}}. \label{eq:ritz-Hm1}
\end{align}
\end{lemma}

\begin{proof}
The $H^2$ estimate follows from C\'ea's lemma applied to the coercive
bilinear form $a(\cdot,\cdot)$ together with standard approximation theory
for $C^1$ finite elements. The $L^2$ and $H^1$ estimates follow by
Aubin--Nitsche duality applied to the periodic operator $\Delta^2+I$, which
has full $H^4$ elliptic regularity on the torus. For the $H^{-1}$ estimate,
the Aubin--Nitsche argument with dual data in $H^{-1}$ yields a dual
solution of regularity $H^3$, which is approximated in $H^2$ at order
$h^{\min(p-1,\,2)}$. The resulting rate is therefore
$h^{p-1}\cdot h^{\min(p-1,2)}=h^{\min(p+2,\,2p-2)}$, which equals $h^{p+2}$
for $p\ge 4$ and $h^4$ for $p=3$.
\end{proof}

\section{Continuous problem}\label{sec:continuous}


We consider the constant-mobility Cahn--Hilliard equation
\begin{equation}\label{eq:CH}
\partial_t\phi=\Delta\bigl(f'(\phi)-\gamma\Delta\phi\bigr)
\qquad\text{in }\Omega\times(0,T],
\end{equation}
where $\gamma>0$ and $f$ is a double-well potential. The continuous Cahn--Hilliard energy is
\begin{equation}\label{eq:E-cont}
E(\phi)
:=
\int_\Omega
f(\phi)+\frac{\gamma}{2}|\nabla\phi|^2\,\dd x.
\end{equation}
Throughout the analysis we assume
\begin{equation}\label{eq:f-assumption}
f\in C^3(\mathbb R),
\qquad
\|f''\|_{L^\infty(\mathbb R)}\le f''_\infty,
\qquad
\|f'''\|_{L^\infty(\mathbb R)}\le f'''_\infty.
\end{equation}
This is satisfied, for example, by a standard globally truncated quartic
double-well potential. Under \eqref{eq:f-assumption},
\begin{align}
\|f'(u)-f'(v)\|
&\le f''_\infty\|u-v\|,
\label{eq:fprime-Lip}\\
\|f''(u)-f''(v)\|
&\le f'''_\infty\|u-v\|
\label{eq:fsecond-Lip}
\end{align}
whenever the expressions are well defined.

Let
$
\mu:=f'(\phi)-\gamma\Delta\phi.
$
The standard weak formulation is
\begin{equation}\label{eq:weak-H1}
(\partial_t\phi,v)+(\nabla\mu,\nabla v)=0
\qquad\forall v\in H^1_{\per}(\Omega).
\end{equation}
If $\phi$ is sufficiently smooth, then the equivalent single-field
$H^2$ formulation is
\begin{equation}\label{eq:weak-H2}
(\partial_t\phi,v)+\gamma(\Delta\phi,\Delta v)
-\bigl(f'(\phi),\Delta v\bigr)=0
\qquad\forall v\in H^2_{\per}(\Omega).
\end{equation}

\begin{theorem}[Continuous mass conservation and energy law]\label{thm:continuous-energy}
Let $\phi$ be a sufficiently smooth solution of \eqref{eq:weak-H1}. Then
$$
(\phi(t),1)=(\phi(0),1)
\qquad\forall t\in[0,T],
$$
and
\begin{equation}\label{eq:continuous-energy-law}
\frac{d}{dt}E(\phi(t))
=
-\|\partial_t\phi(t)\|_{H^{-1}}^2
\le0.
\end{equation}
\end{theorem}

\begin{proof}
Choosing $v\equiv1$ in \eqref{eq:weak-H1} gives mass conservation. Since
$\partial_t\phi$ is mean-free, let $w\in\mathring H^1_{\per}(\Omega)$ solve
$$
(\nabla w,\nabla v)=(\partial_t\phi,v)
\qquad\forall v\in\mathring H^1_{\per}(\Omega).
$$
Then
$$
\|\partial_t\phi\|_{H^{-1}}^2=(\partial_t\phi,w).
$$
On the other hand, \eqref{eq:weak-H1} restricted to mean-free test functions
gives
$$
(\nabla(-\mu),\nabla v)=(\partial_t\phi,v)
\qquad\forall v\in\mathring H^1_{\per}(\Omega),
$$
and hence $w=-\mu+\overline\mu$, where $\overline\mu$ is the spatial mean of
$\mu$. Since $\partial_t\phi$ is mean-free,
$$
(\partial_t\phi,w)=-(\partial_t\phi,\mu)+\overline\mu(\partial_t\phi,1)
=-(\partial_t\phi,\mu).
$$
Thus
$$
\|\partial_t\phi\|_{H^{-1}}^2=-(\partial_t\phi,\mu).
$$
Differentiating \eqref{eq:E-cont} gives
$$
\frac{d}{dt}E(\phi)
=
(f'(\phi)-\gamma\Delta\phi,\partial_t\phi)
=
(\mu,\partial_t\phi),
$$
which proves \eqref{eq:continuous-energy-law}.
\end{proof}

\section{Semidiscrete scheme and physical energy defect}
\label{sec:sd-defect}

We discretize the single-field formulation \eqref{eq:weak-H2} directly in
the $H^2$-conforming space $V_h$.

\subsection{Semidiscrete Galerkin scheme}

Given $\phi_h(0)\in V_h$, find $\phi_h:[0,T]\to V_h$ such that
\begin{equation}\label{eq:sd}
(\partial_t\phi_h,v_h)+\gamma(\Delta\phi_h,\Delta v_h)
-\bigl(f'(\phi_h),\Delta v_h\bigr)=0
\qquad\forall v_h\in V_h
\end{equation}
for a.e.\ $t\in(0,T]$.

\begin{lemma}[Mass conservation]\label{lem:sd-mass}
The semidiscrete solution satisfies
$$
(\phi_h(t),1)=(\phi_h(0),1)
\qquad\forall t\in[0,T].
$$
In particular, $\partial_t\phi_h(t)\in\mathring V_h$.
\end{lemma}

\begin{proof}
Choose $v_h\equiv1$ in \eqref{eq:sd}. Since $\Delta1=0$, the claim follows.
\end{proof}

\begin{remark}[Existence of the semidiscrete solution]\label{rem:sd-exist}
For any $\phi_h(0)\in V_h$, the scheme \eqref{eq:sd} is a system of
ordinary differential equations on the finite-dimensional space $V_h$ with
locally Lipschitz right-hand side, since $f\in C^3(\mathbb R)$ with
$f''\in L^\infty(\mathbb R)$ by \eqref{eq:f-assumption}. Standard ODE theory
yields a unique global solution $\phi_h\in C^1([0,T];V_h)$.
\end{remark}

The physical discrete energy is the continuous energy restricted to $V_h$:
\begin{equation}\label{eq:Eh}
E_h(\phi_h)
:=
\int_\Omega
f(\phi_h)+\frac{\gamma}{2}|\nabla\phi_h|^2\,\dd x.
\end{equation}

\subsection{Energy identity with defect}
To express the time derivative of $E_h$, we introduce the discrete chemical
potential
\begin{equation}\label{eq:qh-def}
q_h:=f'(\phi_h)-\gamma\Delta\phi_h,
\end{equation}
which is the natural variational derivative of $E_h$ at $\phi_h$. A direct
computation gives
\begin{equation}\label{eq:Eh-derivative}
\frac{d}{dt}E_h(\phi_h)=(q_h,\partial_t\phi_h).
\end{equation}
By \Cref{lem:sd-mass}, the time derivative $g_h:=\partial_t\phi_h$ lies in
$\mathring V_h$, so we can compare it with its discrete elliptic
reconstruction $w_h:=(-\Delta_h)^{-1}g_h\in\mathring V_h$. The mismatch
between $-\Delta w_h$ and $g_h$ is the physical reconstruction defect
\begin{equation}\label{eq:physical-rdef}
r_h:=\Delta w_h+g_h.
\end{equation}
At the continuous level, the analogous quantity vanishes identically; on
$V_h$ it is generally nonzero, and the energy defect $\mathcal R_h$ below
quantifies its contribution to the energy balance.

\begin{lemma}[Orthogonality of the physical defect]\label{lem:physical-rorth}
For a.e.\ $t\in(0,T]$, it holds
$$
(r_h,v_h)=0
\qquad\forall v_h\in\mathring V_h.
$$
Consequently, it gives
\begin{equation}\label{eq:physical-projection}
(q_h,r_h)=\bigl((I-\Pi_h^0)q_h,r_h\bigr).
\end{equation}
\end{lemma}

\begin{proof}
By definition of $w_h$, we deduce
$$
(\nabla w_h,\nabla v_h)=(g_h,v_h)
\qquad\forall v_h\in\mathring V_h.
$$
Periodic integration by parts gives
$$
-(\Delta w_h,v_h)=(g_h,v_h)
\qquad\forall v_h\in\mathring V_h,
$$
and hence $(\Delta w_h+g_h,v_h)=0$ for all $v_h\in\mathring V_h$. Note that
$\Delta w_h$ is mean-free by periodic integration by parts and $g_h$ is
mean-free by \Cref{lem:sd-mass}, so $r_h\in\mathring V_h\subset L^2(\Omega)$
is itself mean-free. Therefore, for any $u\in L^2(\Omega)$,
$$
(u,r_h)=\bigl((I-\Pi_h^0)u,r_h\bigr)+(\Pi_h^0 u,r_h)
=\bigl((I-\Pi_h^0)u,r_h\bigr),
$$
since $\Pi_h^0u\in\mathring V_h$ and $(\Pi_h^0u,r_h)=0$ by the orthogonality
just proved. Applied with $u=q_h$, this gives
\eqref{eq:physical-projection}.
\end{proof}

\begin{theorem}[Discrete energy identity with defect]\label{thm:physical-defect}
Let $\phi_h$ solve \eqref{eq:sd}. Then, for a.e.\ $t\in(0,T]$, it yields
\begin{equation}\label{eq:physical-energy-defect}
\frac{d}{dt}E_h(\phi_h(t))
=
-\|\partial_t\phi_h(t)\|_{H^{-1}_h}^2
+\mathcal R_h(t),
\qquad
\mathcal R_h(t):=(q_h(t),r_h(t)).
\end{equation}
Moreover, it holds
\begin{equation}\label{eq:physical-defect-bound}
|\mathcal R_h(t)|
\le
\|(I-\Pi_h^0)q_h(t)\|\,\|r_h(t)\|.
\end{equation}
\end{theorem}

\begin{proof}
Testing \eqref{eq:sd} with $v_h=w_h$ yields
$$
(\partial_t\phi_h,w_h)+\gamma(\Delta\phi_h,\Delta w_h)
-\bigl(f'(\phi_h),\Delta w_h\bigr)=0.
$$
Using $q_h=f'(\phi_h)-\gamma\Delta\phi_h$, this becomes
$$
\|\partial_t\phi_h\|_{H^{-1}_h}^2-(q_h,\Delta w_h)=0.
$$
Since $\Delta w_h=-\partial_t\phi_h+r_h$, we find
$$
(q_h,\partial_t\phi_h)
=
-\|\partial_t\phi_h\|_{H^{-1}_h}^2+(q_h,r_h).
$$
Together with \eqref{eq:Eh-derivative}, this proves
\eqref{eq:physical-energy-defect}. The bound follows from
\eqref{eq:physical-projection} and Cauchy's inequality.
\end{proof}

\begin{remark}[Laplacian-invariant spaces]\label{rem:lap-invariant}
If $\Delta(\mathring V_h)\subset\mathring V_h$ and
$(-\Delta_h)^{-1}$ coincides with the restriction of the continuous inverse
Laplacian to $\mathring V_h$, then $r_h\equiv0$ and hence
$\mathcal R_h\equiv0$. This applies to Fourier or trigonometric polynomial
spaces on periodic domains. For classical globally $C^1$ finite element
spaces, $r_h$ is generally nonzero, and the term $\mathcal R_h$ measures the
failure of exact energy dissipation for the physical discrete energy.
\end{remark}

\section{Relative-energy stability framework}
\label{sec:relative}

We now develop the relative-energy framework that will be used for the error
analysis. For related relative-energy estimates for Cahn--Hilliard systems, we refer to \cite{brunk2023stability,brunk2025analysis}. 

\subsection{Perturbed semidiscrete trajectory}

Let $\phi_h$ solve \eqref{eq:sd}. We compare $\phi_h$ with a perturbed
trajectory $\widehat\phi_h:[0,T]\to V_h$ satisfying
\begin{equation}\label{eq:perturbed}
(\partial_t\widehat\phi_h,v_h)
+\gamma(\Delta\widehat\phi_h,\Delta v_h)
-\bigl(f'(\widehat\phi_h),\Delta v_h\bigr)
=
(\rho_h,v_h)
\qquad\forall v_h\in V_h,
\end{equation}
where $\rho_h(t)\in V_h$ is a residual.
We assume compatibility of means:
\begin{equation}\label{eq:mean-compatible}
(\phi_h(0)-\widehat\phi_h(0),1)=0,
\qquad
(\rho_h(t),1)=0
\quad\text{for a.e. }t\in(0,T].
\end{equation}
Then we define the difference
$
e_h:=\phi_h-\widehat\phi_h\in\mathring V_h
$ for any
$t\in[0,T]$.
Further, we introduce the abbreviations
$g_h:=\partial_t e_h$ and
$F_h:=f'(\phi_h)-f'(\widehat\phi_h)$.
Subtracting \eqref{eq:perturbed} from \eqref{eq:sd} gives
\begin{equation}\label{eq:error-equation}
(g_h,v_h)+\gamma(\Delta e_h,\Delta v_h)
-\bigl(F_h,\Delta v_h\bigr)
=
-(\rho_h,v_h)
\qquad\forall v_h\in V_h.
\end{equation}

\subsection{Relative energies}

The basic $L^2$ relative energy is denoted by
\begin{equation}\label{eq:L2-relative}
\mathcal E_0(\phi_h\mid\widehat\phi_h)
:=
\frac12\|e_h\|^2.
\end{equation}
For $\alpha>0$, we define the $H^1$ relative energy
\begin{equation}\label{eq:H1-relative}
\mathcal E_\alpha(\phi_h\mid\widehat\phi_h)
:=
\frac{\gamma}{2}\|\nabla e_h\|^2
+
\int_\Omega
\Bigl(
f(\phi_h)-f(\widehat\phi_h)-f'(\widehat\phi_h)e_h
\Bigr)\,\dd x
+
\frac{\alpha}{2}\|e_h\|^2.
\end{equation}

\begin{lemma}[Coercivity of $\mathcal E_\alpha$]\label{lem:Ealpha-coercive}
Assume \eqref{eq:f-assumption}. Then there exists $\alpha_0>0$ such that for
every $\alpha\ge\alpha_0$,
\begin{equation}\label{eq:Ealpha-coercive}
c\|e_h\|_{H^1}^2
\le
\mathcal E_\alpha(\phi_h\mid\widehat\phi_h)
\le
C\|e_h\|_{H^1}^2.
\end{equation}
\end{lemma}

\begin{proof}
By Taylor's theorem,
$$
f(\phi_h)-f(\widehat\phi_h)-f'(\widehat\phi_h)e_h
=
\frac12 f''(\widehat\phi_h+\theta e_h)e_h^2
$$
for some $\theta=\theta(x,t)\in(0,1)$. Hence this so-called Bregman term is bounded
above and below by constants times $\|e_h\|^2$. Choosing
$\alpha>f''_\infty$ gives the lower bound; the upper bound is immediate.
\end{proof}

To close the relative-defect estimate, we introduce the higher-order
correction
\begin{equation}\label{eq:H-correction}
\mathcal H_h(\phi_h\mid\widehat\phi_h)
:=
\frac{\gamma}{2}\|\Delta e_h\|^2-(F_h,\Delta e_h)
\end{equation}
and the augmented relative energy
\begin{equation}\label{eq:augmented-energy}
\mathcal E_{\alpha,\beta}(\phi_h\mid\widehat\phi_h)
:=
\mathcal E_\alpha(\phi_h\mid\widehat\phi_h)
+\beta\,\mathcal H_h(\phi_h\mid\widehat\phi_h),
\qquad \beta>0.
\end{equation}

\begin{lemma}[Coercivity of the augmented energy]\label{lem:aug-coercive}
Let $\alpha\ge\alpha_0$. There exists $\beta_0>0$ such that for all
$\beta\in(0,\beta_0]$,
\begin{equation}\label{eq:aug-coercive}
c\Bigl(\|e_h\|_{H^1}^2+\beta\|\Delta e_h\|^2\Bigr)
\le
\mathcal E_{\alpha,\beta}(\phi_h\mid\widehat\phi_h)
\le
C\Bigl(\|e_h\|_{H^1}^2+\beta\|\Delta e_h\|^2\Bigr).
\end{equation}
The constants may depend on $\alpha,\beta,\gamma$, and $f''_\infty$, but not
on $h$.
\end{lemma}

\begin{proof}
By \eqref{eq:fprime-Lip}, it holds
$$
\|F_h\|\le f''_\infty\|e_h\|.
$$
Thus, we deduce that
$$
|(F_h,\Delta e_h)|
\le
\frac{\gamma}{4}\|\Delta e_h\|^2+C\|e_h\|^2,
$$
from which we further conclude
$$
\frac{\gamma}{4}\|\Delta e_h\|^2-C\|e_h\|^2 \le
\mathcal H_h(\phi_h\mid\widehat\phi_h)
\le
C\bigl(\|\Delta e_h\|^2+\|e_h\|^2\bigr).
$$
Combining these estimates with \Cref{lem:Ealpha-coercive} and choosing
$\beta>0$ sufficiently small proves the claim.
\end{proof}

\subsection{Relative defect}

We recall the discrete elliptic
reconstruction
$
w_h:=(-\Delta_h)^{-1}g_h\in\mathring V_h.
$
Then we define the relative reconstruction defect 
\begin{equation}\label{eq:rrel-def}
r_h^{\rm rel}:=\Delta w_h+g_h
\end{equation}
and the relative energy defect
\begin{equation}\label{eq:Drel-def}
\mathcal D_h(\phi_h\mid\widehat\phi_h)
:=
\bigl(F_h-\gamma\Delta e_h,\ r_h^{\rm rel}\bigr).
\end{equation}
Furthermore, we set
\begin{equation}\label{eq:Taylor-term}
\mathcal T_h(\phi_h\mid\widehat\phi_h)
:=
\bigl(F_h-f''(\widehat\phi_h)e_h,\ \partial_t\widehat\phi_h\bigr).
\end{equation}

\begin{lemma}[Orthogonality of the relative reconstruction defect]\label{lem:rrel-orth}
For a.e.\ $t\in(0,T]$, it holds
$$
(r_h^{\rm rel},v_h)=0
\qquad\forall v_h\in\mathring V_h.
$$
Consequently, it yields
\begin{equation}\label{eq:Drel-proj}
\mathcal D_h(\phi_h\mid\widehat\phi_h)
=
\bigl((I-\Pi_h^0)(F_h-\gamma\Delta e_h),\ r_h^{\rm rel}\bigr).
\end{equation}
\end{lemma}

\begin{proof}
Since we have
$$
(\nabla w_h,\nabla v_h)=(g_h,v_h)
\qquad\forall v_h\in\mathring V_h,
$$
periodic integration by parts gives
$$
-(\Delta w_h,v_h)=(g_h,v_h)
\qquad\forall v_h\in\mathring V_h,
$$
so $(\Delta w_h+g_h,v_h)=0$ for all $v_h\in\mathring V_h$. Both $\Delta w_h$
(by periodic integration by parts) and $g_h=\partial_t e_h$ (since
$e_h\in\mathring V_h$ by \eqref{eq:mean-compatible}) are mean-free, hence
$r_h^{\rm rel}\in\mathring V_h$ is mean-free as well. Therefore
$\Pi_h^0(F_h-\gamma\Delta e_h)\in\mathring V_h$ is orthogonal to
$r_h^{\rm rel}$, and \eqref{eq:Drel-proj} follows.
\end{proof}

We prove the following lemma, giving structural properties for the $L^2$ and $H^1$ relative energies.

\begin{lemma}[Relative-energy identities]\label{lem:H1-identity}
For a.e.\ $t\in(0,T]$, it holds
\begin{equation}\label{eq:L2-identity}
\frac12\frac{d}{dt}\|e_h\|^2
+\gamma\|\Delta e_h\|^2
=
(F_h,\Delta e_h)-(\rho_h,e_h),
\end{equation}
and
\begin{align}
\frac{d}{dt}\mathcal E_\alpha(\phi_h\mid\widehat\phi_h)
&+
\|g_h\|_{H^{-1}_h}^2
+\alpha\gamma\|\Delta e_h\|^2 \notag\\
&=
\alpha(F_h,\Delta e_h)
-(\rho_h,w_h+\alpha e_h)
+\mathcal D_h(\phi_h\mid\widehat\phi_h)
+\mathcal T_h(\phi_h\mid\widehat\phi_h).
\label{eq:H1-identity}
\end{align}
\end{lemma}

\begin{proof}
Choosing $v_h=e_h$ in \eqref{eq:error-equation} gives \eqref{eq:L2-identity}.
Further, differentiating \eqref{eq:H1-relative} yields
$$
\frac{d}{dt}\mathcal E_\alpha(\phi_h\mid\widehat\phi_h)
=
(F_h-\gamma\Delta e_h,g_h)
+\alpha(e_h,g_h)
+\mathcal T_h(\phi_h\mid\widehat\phi_h).
$$
Testing \eqref{eq:error-equation} with $v_h=w_h$ gives
$$
(g_h,w_h)+\gamma(\Delta e_h,\Delta w_h)-(F_h,\Delta w_h)=-(\rho_h,w_h).
$$
Since $(g_h,w_h)=\|g_h\|_{H^{-1}_h}^2$ and
$\Delta w_h=-g_h+r_h^{\rm rel}$, we obtain
$$
(F_h-\gamma\Delta e_h,g_h)
=
-\|g_h\|_{H^{-1}_h}^2
-(\rho_h,w_h)
+\mathcal D_h(\phi_h\mid\widehat\phi_h).
$$
Lastly, we test \eqref{eq:error-equation} with $v_h=\alpha e_h$ and deduce
$$
\alpha(e_h,g_h)
=
-\alpha\gamma\|\Delta e_h\|^2
+\alpha(F_h,\Delta e_h)
-\alpha(\rho_h,e_h).
$$
Combining these identities proves \eqref{eq:H1-identity}.
\end{proof}

\begin{lemma}[Auxiliary identity]\label{lem:aux-identity}
For a.e.\ $t\in(0,T]$,
\begin{equation}\label{eq:aux-identity}
\frac{d}{dt}\mathcal H_h(\phi_h\mid\widehat\phi_h)+\|g_h\|^2
=
-(\rho_h,g_h)-(\partial_tF_h,\Delta e_h).
\end{equation}
\end{lemma}

\begin{proof}
Choose $v_h=g_h$ in \eqref{eq:error-equation}. Then
$$
\|g_h\|^2+\gamma(\Delta e_h,\Delta g_h)-(F_h,\Delta g_h)=-(\rho_h,g_h).
$$
Then using
$
(\Delta e_h,\Delta g_h)=\frac12\frac{d}{dt}\|\Delta e_h\|^2
$
and
$$
(F_h,\Delta g_h)=\frac{d}{dt}(F_h,\Delta e_h)-(\partial_tF_h,\Delta e_h),
$$
we deuce \eqref{eq:aux-identity}.
\end{proof}

The defect and remainder estimates below rely on a discrete elliptic
reconstruction bound for the inverse discrete Laplacian, stated as
\Cref{prop:reconstruction} in \Cref{app:reconstruction}.

\subsection{Defect and remainder estimates}

\begin{lemma}[Relative defect estimate]\label{lem:Drel-estimate}
Assume \eqref{eq:f-assumption}. Then, for every $\varepsilon>0$,
\begin{equation}\label{eq:Drel-estimate}
|\mathcal D_h(\phi_h\mid\widehat\phi_h)|
\le
\varepsilon\|g_h\|^2
+
C_\varepsilon
\left(
\mathcal E_\alpha(\phi_h\mid\widehat\phi_h)
+\|\Delta e_h\|^2
\right).
\end{equation}
\end{lemma}

\begin{proof}
Using \eqref{eq:Drel-proj}, we split $\mathcal D_h$ into two parts as follows
$$
\mathcal D_h
=
\bigl((I-\Pi_h^0)F_h,r_h^{\rm rel}\bigr)
-
\gamma\bigl((I-\Pi_h^0)\Delta e_h,r_h^{\rm rel}\bigr)
=:D_1+D_2.
$$
For the nonlinear part $D_1$, we note that $I-\Pi_h^0$ is an $L^2$-orthogonal projection and
hence has operator norm at most one, so $\|(I-\Pi_h^0)F_h\|\le\|F_h\|$. The
Lipschitz bound \eqref{eq:fprime-Lip} on $f'$ gives
$$\|F_h\|
=
\|f'(\phi_h)-f'(\widehat\phi_h)\|
\le
\|f''\|_{L^\infty}\|\phi_h-\widehat\phi_h\|
=
C\|e_h\|,$$
so in total
$$
\|(I-\Pi_h^0)F_h\|\le \|F_h\|\le f''_\infty\|e_h\|.
$$
Further, since
$r_h^{\rm rel}=\Delta w_h+g_h$ and
$w_h=(-\Delta_h)^{-1}g_h$, \Cref{prop:reconstruction} implies $\|r_h^{\rm rel}\|\le C\|g_h\|$.
Thus, in total,
$$
|D_1|\le C\|e_h\|\|g_h\|
\le
\varepsilon\|g_h\|^2+C_\varepsilon\|e_h\|^2.
$$
For the linear part $D_2$, we deduce
$$
|D_2|
\le
C\|\Delta e_h\|\|g_h\|
\le
\varepsilon\|g_h\|^2+C_\varepsilon\|\Delta e_h\|^2.
$$
Finally, it holds $\|e_h\|^2\le C\mathcal E_\alpha(\phi_h\mid\widehat\phi_h)$ by
\Cref{lem:Ealpha-coercive}.
\end{proof}

\begin{lemma}[Taylor and commutator estimates]\label{lem:Taylor-comm}
Assume \eqref{eq:f-assumption} and
\begin{equation}\label{eq:comparison-bounds}
\widehat\phi_h\in L^\infty(0,T;H^2(\Omega)),
\qquad
\partial_t\widehat\phi_h\in L^\infty(0,T;H^1(\Omega)).
\end{equation}
Then, for every $\varepsilon>0$,
\begin{align}
|\mathcal T_h(\phi_h\mid\widehat\phi_h)|
&\le
C\mathcal E_\alpha(\phi_h\mid\widehat\phi_h), \label{eq:Taylor-est}\\
|(\partial_tF_h,\Delta e_h)|
&\le
\varepsilon\|g_h\|^2
+
C_\varepsilon
\left(
\mathcal E_\alpha(\phi_h\mid\widehat\phi_h)
+\|\Delta e_h\|^2
\right).
\label{eq:comm-est}
\end{align}
\end{lemma}

\begin{proof}
Taylor's theorem and \eqref{eq:f-assumption} imply
$$
|F_h-f''(\widehat\phi_h)e_h|\le C|e_h|^2.
$$
Using the embedding $H^1(\Omega)\hookrightarrow L^4(\Omega)$, we obtain
$$
\|F_h-f''(\widehat\phi_h)e_h\|
\le C\|e_h\|_{L^4}^2
\le C\|e_h\|_{H^1}^2.
$$
Hence, it yields \eqref{eq:Taylor-est} as follows
$$
|\mathcal T_h|
\le
C\|e_h\|_{H^1}^2
\le
C\mathcal E_\alpha(\phi_h\mid\widehat\phi_h).
$$

Next, we use the definition of $F_h$ to write
$$
\partial_tF_h
=
f''(\phi_h)g_h
+
\bigl(f''(\phi_h)-f''(\widehat\phi_h)\bigr)\partial_t\widehat\phi_h.
$$
Using \eqref{eq:f-assumption}, $H^1(\Omega)\hookrightarrow L^4(\Omega)$, and
\eqref{eq:comparison-bounds}, we deduce
$$
\|\partial_tF_h\|
\le
C\|g_h\|+C\|e_h\|_{L^4}\|\partial_t\widehat\phi_h\|_{L^4}
\le
C\|g_h\|+C\|e_h\|_{H^1}.
$$
Therefore, it yields
$$
|(\partial_tF_h,\Delta e_h)|
\le
C\|g_h\|\|\Delta e_h\|
+
C\|e_h\|_{H^1}\|\Delta e_h\|,
$$
and Young's inequality gives \eqref{eq:comm-est}.
\end{proof}

\subsection{Closed augmented stability estimate}

\begin{theorem}[Augmented relative-energy stability]\label{thm:aug-stability}
Assume \eqref{eq:f-assumption} and \eqref{eq:comparison-bounds}. Let
$\alpha\ge\alpha_0$ and choose a fixed $\beta\in(0,\beta_0]$. Then there exist constants $c_1,c_2,c_3,C>0$,
independent of $h$, such that
\begin{align}
\frac{d}{dt}\mathcal E_{\alpha,\beta}(\phi_h\mid\widehat\phi_h)
&+
c_1\|g_h\|_{H^{-1}_h}^2
+
c_2\|g_h\|^2
+
c_3\|\Delta e_h\|^2 \notag\\
&\le
C\mathcal E_{\alpha,\beta}(\phi_h\mid\widehat\phi_h)
+
C\|\rho_h\|_{H^{-1}}^2
+
C\|\rho_h\|^2.
\label{eq:aug-stability}
\end{align}
Consequently, it holds
\begin{align}
\mathcal E_{\alpha,\beta}(t)
&+
c_1\int_0^t\|g_h(s)\|_{H^{-1}_h}^2\,ds
+
c_2\int_0^t\|g_h(s)\|^2\,ds
+
c_3\int_0^t\|\Delta e_h(s)\|^2\,ds \notag\\
&\le
C_T\left(
\mathcal E_{\alpha,\beta}(0)
+
\int_0^t\|\rho_h(s)\|_{H^{-1}}^2\,ds
+
\int_0^t\|\rho_h(s)\|^2\,ds
\right).
\label{eq:aug-stability-int}
\end{align}
\end{theorem}

\begin{proof}
We fix the order of constants as follows: choose $\alpha\ge\alpha_0$ first
(this fixes the coercivity constants of $\mathcal E_\alpha$), then choose
$\beta\in(0,\beta_0(\alpha)]$ from \Cref{lem:aug-coercive} (this fixes the
coercivity constants of $\mathcal E_{\alpha,\beta}$), and finally choose
the absorption parameters $\varepsilon$ in
\Cref{lem:Drel-estimate,lem:Taylor-comm} small relative to $\beta$ and
$\alpha\gamma$.
Adding \eqref{eq:H1-identity} and $\beta$ times \eqref{eq:aux-identity}
gives
\begin{align*}
\frac{d}{dt}\mathcal E_{\alpha,\beta}
&+
\|g_h\|_{H^{-1}_h}^2
+
\beta\|g_h\|^2
+
\alpha\gamma\|\Delta e_h\|^2 \\
&=
\alpha(F_h,\Delta e_h)
-(\rho_h,w_h+\alpha e_h+\beta g_h)
+\mathcal D_h
+\mathcal T_h
-\beta(\partial_tF_h,\Delta e_h).
\end{align*}
The term $\alpha(F_h,\Delta e_h)$ is bounded by
$$
\alpha|(F_h,\Delta e_h)|
\le
\frac{\alpha\gamma}{4}\|\Delta e_h\|^2
+
C\|e_h\|^2.
$$
The residual terms satisfy
$$
\begin{aligned}
|(\rho_h,w_h)|
&\le
\|\rho_h\|_{H^{-1}}\|g_h\|_{H^{-1}_h}
\le
\frac14\|g_h\|_{H^{-1}_h}^2+C\|\rho_h\|_{H^{-1}}^2, \\
\alpha|(\rho_h,e_h)|
&\le
C\|\rho_h\|_{H^{-1}}\|e_h\|_{H^1}
\le
C\|\rho_h\|_{H^{-1}}^2+C\mathcal E_{\alpha,\beta}, \\
\beta|(\rho_h,g_h)|
&\le
\frac{\beta}{4}\|g_h\|^2+C\|\rho_h\|^2.
\end{aligned}
$$
Using \Cref{lem:Drel-estimate,lem:Taylor-comm}, with the order of constants
fixed above, the $\|g_h\|^2$ and $\|\Delta e_h\|^2$ terms can be absorbed
into the left-hand side. This proves \eqref{eq:aug-stability}. The
integrated estimate follows by Gronwall's lemma.
\end{proof}

\section{Application to the exact solution}
\label{sec:error}

We now apply the relative-energy framework to the exact solution. Let
$\phi$ be a sufficiently smooth solution of \eqref{eq:weak-H2}. We choose
$
\widehat\phi_h:=R_h\phi$ and 
$\eta_h:=\widehat\phi_h-\phi$.
By the mass-preservation property \eqref{eq:ritz-mean} of the Ritz
projection, $\eta_h$ is mean-free.
The residual associated with $\widehat\phi_h$ is defined by
\begin{equation}\label{eq:rho-def}
(\rho_h,v_h)
:=
(\partial_t\widehat\phi_h,v_h)
+\gamma(\Delta\widehat\phi_h,\Delta v_h)
-\bigl(f'(\widehat\phi_h),\Delta v_h\bigr)
\qquad\forall v_h\in V_h.
\end{equation}

\begin{lemma}[Residual representation]\label{lem:rho-repr}
For all $v_h\in V_h$, it holds
\begin{equation}\label{eq:rho-repr}
(\rho_h,v_h)
=
(\partial_t\eta_h,v_h)
-\gamma(\eta_h,v_h)
-
\bigl(f'(\widehat\phi_h)-f'(\phi),\Delta v_h\bigr).
\end{equation}
\end{lemma}

\begin{proof}
Since $\phi$ solves \eqref{eq:weak-H2},
$$
(\partial_t\phi,v_h)
+\gamma(\Delta\phi,\Delta v_h)
-\bigl(f'(\phi),\Delta v_h\bigr)=0.
$$
Subtracting this from \eqref{eq:rho-def} gives
$$
(\rho_h,v_h)
=
(\partial_t\eta_h,v_h)
+\gamma(\Delta\eta_h,\Delta v_h)
-\bigl(f'(\widehat\phi_h)-f'(\phi),\Delta v_h\bigr).
$$
The Ritz orthogonality \eqref{eq:ritz} applied to $v=\phi$ gives, for all
$v_h\in V_h$,
$$
(\Delta\eta_h,\Delta v_h)+(\eta_h,v_h)=0,
\qquad\text{i.e.\ }\quad
(\Delta\eta_h,\Delta v_h)=-(\eta_h,v_h).
$$
Substituting yields \eqref{eq:rho-repr}.
\end{proof}

\subsection{Residual bounds}

\begin{lemma}[Residual estimates]\label{lem:rho-estimates}
Assume \eqref{eq:f-assumption}. Then
\begin{align}
\|\rho_h\|_{H^{-2}_h}
&\le
C\left(
\|\partial_t\eta_h\|_{H^{-2}}
+\|\eta_h\|
\right), \label{eq:rho-Hm2}\\
\|\rho_h\|_{H^{-1}}
&\le
C\left(
\|\partial_t\eta_h\|_{H^{-1}}
+h\|\partial_t\eta_h\|
+h^{-1}\|\eta_h\|
\right), \label{eq:rho-Hm1}\\
\|\rho_h\|
&\le
C\left(
\|\partial_t\eta_h\|
+h^{-2}\|\eta_h\|
\right). \label{eq:rho-L2}
\end{align}
\end{lemma}

\begin{proof}
First, we use \eqref{eq:rho-repr} and
\eqref{eq:fprime-Lip} as follows:
$$
\begin{aligned}
|(\rho_h,v_h)|
&\le
\|\partial_t\eta_h\|_{H^{-2}}\|v_h\|_{H^2}
+\gamma\|\eta_h\|\|v_h\|
+
C\|\eta_h\|\|\Delta v_h\|
\\ &\le
C\bigl(\|\partial_t\eta_h\|_{H^{-2}}+\|\eta_h\|\bigr)\|v_h\|_{H^2}.
\end{aligned}
$$
Taking the supremum over $v_h\in\mathring V_h$ proves \eqref{eq:rho-Hm2}.
For \eqref{eq:rho-L2}, since $\rho_h\in V_h$, we have
$$
\|\rho_h\|=
\sup_{0\ne v_h\in V_h}\frac{(\rho_h,v_h)}{\|v_h\|}.
$$
Using \eqref{eq:rho-repr} and the inverse estimate
$\|\Delta v_h\|\le Ch^{-2}\|v_h\|$ gives \eqref{eq:rho-L2}; the
$\gamma(\eta_h,v_h)$ term contributes at most $C\|\eta_h\|\|v_h\|$, which is dominated by the $h^{-2}\|\eta_h\|$ contribution in \eqref{eq:rho-L2}.

For \eqref{eq:rho-Hm1}, take arbitrary $\psi\in H^1_{\per}(\Omega)$ and
write
$$
(\rho_h,\psi)
=
(\rho_h,\psi-\mathcal I_h\psi)+(\rho_h,\mathcal I_h\psi).
$$
The first term is bounded by
$$
|(\rho_h,\psi-\mathcal I_h\psi)|
\le
\|\rho_h\|\|\psi-\mathcal I_h\psi\|
\le
Ch\|\rho_h\|\|\psi\|_{H^1}.
$$
For the second term, \eqref{eq:rho-repr} gives
$$
|(\rho_h,\mathcal I_h\psi)|
\le
\|\partial_t\eta_h\|_{H^{-1}}\|\mathcal I_h\psi\|_{H^1}
+\gamma\|\eta_h\|\|\mathcal I_h\psi\|
+
C\|\eta_h\|\|\Delta\mathcal I_h\psi\|.
$$
Using \eqref{eq:I-H1-stab} and \eqref{eq:I-lap-H1}, we deduce that
$$
|(\rho_h,\mathcal I_h\psi)|
\le
C\bigl(\|\partial_t\eta_h\|_{H^{-1}}+h^{-1}\|\eta_h\|\bigr)\|\psi\|_{H^1}.
$$
Combining this with \eqref{eq:rho-L2} proves \eqref{eq:rho-Hm1}.
\end{proof}

\begin{theorem}[Residual rates]\label{thm:rho-rates}
Assume
\begin{equation}\label{eq:regularity}
\phi\in L^\infty(0,T;H^{p+1}(\Omega)),
\qquad
\partial_t\phi\in L^2(0,T;H^{p+1}(\Omega))\cap L^\infty(0,T;H^2(\Omega)).
\end{equation}
Then
\begin{align}
\|\rho_h\|_{L^2(0,T;H^{-2}_h)}
&\le Ch^{p+1}, \label{eq:rho-rate-Hm2}\\
\|\rho_h\|_{L^2(0,T;H^{-1})}
&\le Ch^{p}, \label{eq:rho-rate-Hm1}\\
\|\rho_h\|_{L^2(0,T;L^2)}
&\le Ch^{p-1}. \label{eq:rho-rate-L2}
\end{align}
Moreover, $\widehat\phi_h$ satisfies the regularity stated in \eqref{eq:comparison-bounds}.
\end{theorem}

\begin{proof}
Since $\widehat\phi_h=R_h\phi$, the Ritz estimates of
\Cref{lem:ritz-error} imply
$$
\|\eta_h\|_{L^\infty(0,T;L^2)}\le Ch^{p+1},
\qquad
\|\eta_h\|_{L^\infty(0,T;H^1)}\le Ch^p,
\qquad
\|\eta_h\|_{L^\infty(0,T;H^2)}\le Ch^{p-1}.
$$
Similarly, since $R_h$ commutes with $\partial_t$, it holds
$$
\|\partial_t\eta_h\|_{L^2(0,T;L^2)}\le Ch^{p+1},
\qquad
\|\partial_t\eta_h\|_{L^2(0,T;H^{-1})}\le Ch^{\min(p+2,\,2p-2)}.
$$
Since $\|\partial_t\eta_h\|_{H^{-2}}\le \|\partial_t\eta_h\|$, we also have
$$
\|\partial_t\eta_h\|_{L^2(0,T;H^{-2})}\le Ch^{p+1}.
$$
Inserting these bounds into \Cref{lem:rho-estimates} gives
\eqref{eq:rho-rate-Hm2}--\eqref{eq:rho-rate-L2}; note that the $H^{-1}$
residual rate $h^p$ is unaffected by the BFS-specific $h^4$ rate of
$\|\partial_t\eta_h\|_{L^2(H^{-1})}$, since the $h^{-1}\|\eta_h\|=h^p$
contribution dominates in all cases.

It remains to verify \eqref{eq:comparison-bounds}. By the triangle
inequality, it holds
$$
\|\widehat\phi_h\|_{L^\infty(0,T;H^2)}
\le
\|\eta_h\|_{L^\infty(0,T;H^2)}
+\|\phi\|_{L^\infty(0,T;H^2)}
\le C,
$$
using the regularity \eqref{eq:regularity}. The same argument applied to
$\partial_t\widehat\phi_h=R_h\partial_t\phi$ together with
$\partial_t\phi\in L^\infty(0,T;H^2)$ yields
$$
\|\partial_t\widehat\phi_h\|_{L^\infty(0,T;H^1)}\le C,
$$
proving \eqref{eq:comparison-bounds}.
\end{proof}

\subsection{Error estimates}

We denote the difference as before by
$
e_h:=\phi_h-\widehat\phi_h.
$
We choose the initial value
\begin{equation}\label{eq:initial}
\phi_h(0)=R_h\phi(0),
\end{equation}
so that it holds $e_h(0)=0$.

\begin{lemma}[Discrete $L^2$ estimate]\label{lem:L2-error}
Under the assumptions of \Cref{thm:rho-rates}, it gives
\begin{equation}\label{eq:L2-error-discrete}
\|e_h\|_{L^\infty(0,T;L^2)}
+
\|e_h\|_{L^2(0,T;H^2)}
\le
Ch^{p+1}.
\end{equation}
\end{lemma}

\begin{proof}
From \Cref{lem:H1-identity}, \eqref{eq:fprime-Lip}, and Young's inequality, we deduce
$$
\frac12\frac{d}{dt}\|e_h\|^2
+\frac{\gamma}{2}\|\Delta e_h\|^2
\le
C\|e_h\|^2
+C\|\rho_h\|_{H^{-2}_h}^2,
$$
where we used
$$
|(\rho_h,e_h)|
\le
\|\rho_h\|_{H^{-2}_h}\|e_h\|_{H^2}
\le
\frac{\gamma}{4}\|\Delta e_h\|^2
+C\|\rho_h\|_{H^{-2}_h}^2.
$$
Since $e_h(0)=0$, Gronwall's lemma and \eqref{eq:rho-rate-Hm2} imply
$$
\|e_h\|_{L^\infty(0,T;L^2)}^2
+
\int_0^T\|\Delta e_h\|^2\,dt
\le
Ch^{2p+2}.
$$
The equivalence \eqref{eq:poincare-H2} on mean-free functions gives
\eqref{eq:L2-error-discrete}.
\end{proof}

\begin{lemma}[Augmented relative-energy estimate]\label{lem:aug-error}
Under the assumptions of \Cref{thm:rho-rates},
\begin{equation}\label{eq:aug-error-discrete}
\|e_h\|_{L^\infty(0,T;H^2)}
+
\|\partial_t e_h\|_{L^2(0,T;H^{-1}_h)}
+
\|\partial_t e_h\|_{L^2(0,T;L^2)}
+
\|e_h\|_{L^2(0,T;H^2)}
\le
Ch^{p-1}.
\end{equation}
\end{lemma}

\begin{proof}
We use $e_h(0)=0$ again to deduce
$$
\mathcal E_{\alpha,\beta}(\phi_h(0)\mid\widehat\phi_h(0))=0.
$$
Using \Cref{thm:aug-stability} and the residual estimates
\eqref{eq:rho-rate-Hm1}--\eqref{eq:rho-rate-L2}, we obtain
$$
\mathcal E_{\alpha,\beta}(t)
+
\int_0^t\|\partial_t e_h\|_{H^{-1}_h}^2\,ds
+
\int_0^t\|\partial_t e_h\|^2\,ds
+
\int_0^t\|\Delta e_h\|^2\,ds
\le
C h^{2p-2}.
$$
The coercivity of the augmented energy, \Cref{lem:aug-coercive}, gives
$$
\|e_h\|_{L^\infty(0,T;H^1)}
+
\|e_h\|_{L^\infty(0,T;H^2)}
\le
Ch^{p-1},
$$
and the remaining terms follow from the integral estimate.
\end{proof}

\begin{theorem}[Semidiscrete a priori error estimates]\label{thm:main-error}
Assume \eqref{eq:f-assumption}, \eqref{eq:regularity}, and choose the initial
condition \eqref{eq:initial}. Then it yields
\begin{align}
\|\phi-\phi_h\|_{L^\infty(0,T;L^2)}
&\le Ch^{p+1}, \label{eq:main-L2}\\
\|\phi-\phi_h\|_{L^\infty(0,T;H^1)}
&\le Ch^{p}, \label{eq:main-H1}\\
\|\phi-\phi_h\|_{L^\infty(0,T;H^2)}
&\le Ch^{p-1}, \label{eq:main-H2}\\
\|\phi-\phi_h\|_{L^2(0,T;H^2)}
&\le Ch^{p-1}. \label{eq:main-L2H2}
\end{align}
Furthermore, it holds
\begin{equation}\label{eq:main-time-discrete}
\|\partial_t(\widehat\phi_h-\phi_h)\|_{L^2(0,T;H^{-1}_h)}
+
\|\partial_t(\widehat\phi_h-\phi_h)\|_{L^2(0,T;L^2)}
\le
Ch^{p-1}.
\end{equation}
\end{theorem}

\begin{proof}
We split
$$
\phi-\phi_h
=
(\phi-\widehat\phi_h)+(\widehat\phi_h-\phi_h)
=
-\eta_h-e_h.
$$
The projection estimates give
$$
\|\eta_h\|_{L^\infty(0,T;L^2)}\le Ch^{p+1},
\qquad
\|\eta_h\|_{L^\infty(0,T;H^1)}\le Ch^p,
\qquad
\|\eta_h\|_{L^\infty(0,T;H^2)}\le Ch^{p-1}.
$$
Together with \Cref{lem:L2-error,lem:aug-error}, this yields
\eqref{eq:main-L2}, \eqref{eq:main-H2}, and \eqref{eq:main-L2H2}.

It remains to prove the $H^1$ estimate with the optimal order. Since
$\phi-\phi_h$ is mean-free, the interpolation inequality
$$
\|v\|_{H^1}^2\le C\|v\|\,\|v\|_{H^2},
\qquad v\in\mathring H^2_{\per}(\Omega),
$$
gives
$$
\|\phi-\phi_h\|_{L^\infty(0,T;H^1)}^2
\le
C
\|\phi-\phi_h\|_{L^\infty(0,T;L^2)}
\|\phi-\phi_h\|_{L^\infty(0,T;H^2)}
\le
C h^{p+1}h^{p-1}
=
C h^{2p}.
$$
This proves \eqref{eq:main-H1}. Finally,
\eqref{eq:main-time-discrete} is exactly the time-derivative part of
\Cref{lem:aug-error}.
\end{proof}

\begin{remark}[On the role of the higher-order estimates]
For constant mobility, the $H^1$ and $H^2$ error estimates of
\Cref{thm:main-error} are conditional in the sense that they require the
augmented relative-energy framework, while the optimal $L^\infty(0,T;L^2)$
rate already follows from the basic $L^2$ relative-energy estimate of
\Cref{lem:L2-error}. In the non-constant-mobility case the $L^2$
estimate becomes substantially more involved, whereas the higher-order
estimates obtained via the augmented relative energy carry over with
essentially no change.
\end{remark}

\begin{remark}[Comparison with mixed formulations]
It is sometimes argued in the literature that single-field
$H^2$-conforming discretizations are preferable to mixed
$H^1$-conforming formulations for the Cahn--Hilliard equation. The
present analysis points in the opposite direction at the level of
energy stability: an exact discrete dissipation law is, in general,
\emph{not} available for $H^2$-conforming spaces, due to the energy
defect $\mathcal R_h$ identified in \Cref{thm:physical-defect}. By
contrast, energy stability is straightforward in the mixed setting.
\end{remark}

\section{Numerical experiments}
\label{sec:numerics}

We now illustrate the convergence behaviour and the energy-defect mechanism
for the double-well potential
$
  f(\phi)=\frac14\phi^2(1-\phi)^2.
$
All computations are performed on the periodic unit square
\(\Omega=(0,1)^2\). In order to separate the spatial energy defect from a
possible time-discretization defect, we use the averaged-vector-field
treatment of the bulk potential together with a midpoint treatment of the
interfacial term. Thus the fully discrete scheme reads
$$
\frac{1}{\tau}(\phi_h^{n+1}-\phi_h^n,v_h)
+
\gamma\left(\Delta\frac{\phi_h^{n+1}+\phi_h^n}{2},\Delta v_h\right)
-
\bigl(f'_{\rm avf}(\phi_h^{n+1},\phi_h^n),\Delta v_h\bigr)
=0 ,
$$
where  \vspace{-.2cm}
$$
f'_{\rm avf}(b,a)
=
\int_0^1 f'(a+\theta(b-a))\,d\theta
=
\frac14(a^3+a^2b+ab^2+b^3)
-\frac12(a^2+ab+b^2)
+\frac14(a+b).
$$

\subsection{Manufactured-solution convergence}

We first verify the spatial convergence rates using a smooth manufactured
periodic solution. The forcing is chosen at the fully discrete level, so that
the exact smooth reference function satisfies the AVF/midpoint time
discretization. This allows us to use a fixed time step and isolate the
spatial discretization error. We take \(\gamma=5\cdot10^{-3}\),
\(\tau=10^{-3}\), and \(T=5\cdot10^{-3}\).
The errors are measured in
the natural solution spaces, and we additionally report the discrete time-derivative errors in
\(L^2(0,T;H^{-1}_h)\) and \(L^2(0,T;L^2)\). The results are shown in
Figures~\ref{fig:convergence-bell-argyris} and
\ref{fig:time-derivative-convergence}. For the Bell element, we observe rates
close to \(5\), \(4\), and \(3\) in \(L^\infty(0,T;L^2)\), \(L^\infty(0,T;H^1)\), and
\(L^\infty(0,T;H^2)\), respectively. For the Argyris element, the corresponding
observed rates are close to \(6\), \(5\), and \(4\). These rates agree with
the expected approximation orders of \Cref{thm:main-error}.

\begin{figure}[htp!]
\centering
\includegraphics[width=.9\textwidth,page=1]{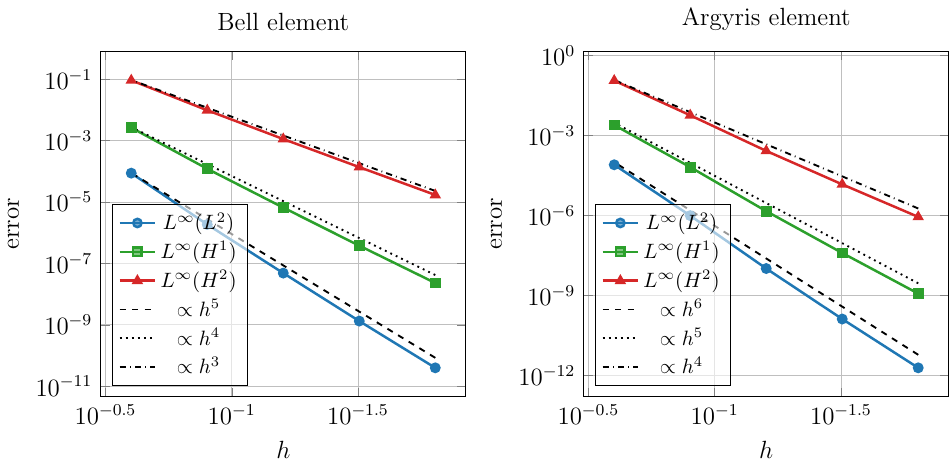}
\caption{Spatial convergence for the manufactured-solution test with fixed
time step \(\tau=10^{-3}\). The observed rates match the expected orders:
\(h^5,h^4,h^3\) for Bell and \(h^6,h^5,h^4\) for Argyris in
\(L^\infty(0,T;L^2)\), \(L^\infty(0,T;H^1)\), and \(L^\infty(0,T;H^2)\), respectively.}
\label{fig:convergence-bell-argyris}
\end{figure}

\begin{figure}[htp!]
\centering
\includegraphics[width=.6\textwidth,page=2]{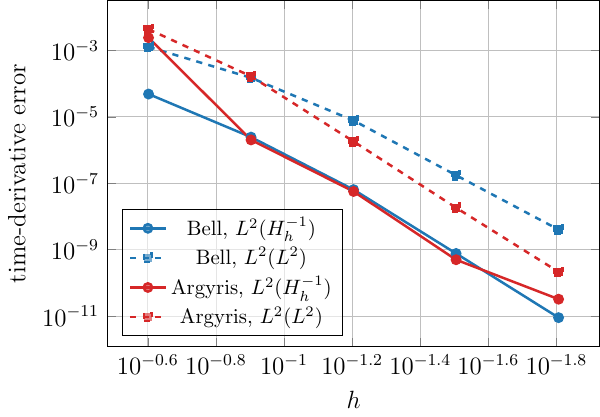}
\caption{Convergence of the discrete time-derivative error in the
manufactured-solution experiment with fixed \(\tau=10^{-3}\).}
\label{fig:time-derivative-convergence}
\end{figure}

\subsection{Energy defect}

We next study the physical energy defect in a spinodal-decomposition-type
experiment. We use \(\gamma=10^{-4}\), \(\tau=10^{-4}\), and \(T=10^{-1}\).
The initial condition is the same low-mode trigonometric perturbation for all
three discretizations, with mean value \(0.4\). This gives almost identical
initial energies for the Fourier, Bell, and Argyris runs.

The Fourier computation uses \(128^2\) modes. Since Fourier spaces are
Laplacian-invariant, the theory predicts a vanishing spatial defect. The
finite element computations use Bell and Argyris elements on the same
\(N=64\) mesh. For the finite element discretizations we evaluate the defect by two
independent procedures: the balance form
$$R_h^{n,\mathrm{bal}}=(E_h^{n+1}-E_h^n)/\tau+D_h^n,$$ with $D_h^n
:=\|\delta_\tau\phi_h^{n+1}\|_{H^{-1}_h}^2$, and the direct form
$$
R_h^{n}:=(q_h^n,r_h^n),$$
with
$q_h^n=f'_{\rm avf}(\phi_h^{n+1},\phi_h^n)
-\gamma\Delta(\phi_h^{n+1}+\phi_h^n)/2$ and 
$r_h^n=\Delta w_h^n+\delta_\tau\phi_h^{n+1}.
$
The two should agree up to roundoff; we use the direct form in the
plots and verify agreement as a consistency check.
In addition to the pointwise ratio
\(|R_h^n|/D_h^n\), we therefore also consider the cumulative relative defect
$$
  \frac{\sum_{k=0}^{n-1}\tau |R_h^k|}
       {\sum_{k=0}^{n-1}\tau D_h^k}\quad \text{ where } D_h^k=\|\delta_\tau\phi_h^{k+1}\|_{H^{-1}_h}^2,
$$
which is less sensitive to sign changes of the pointwise defect.

The energy evolution and the signed defect are displayed in
Figure~\ref{fig:energy-and-signed-defect}. The energy curves are nearly
indistinguishable on this scale, showing that the three discretizations
produce essentially the same physical evolution. The defect plot, however,
separates the Laplacian-invariant Fourier discretization from the classical
\(C^1\) finite element spaces.
The Fourier defect is at roundoff level, with
$$
  \max_n |R^n|=5.4\cdot 10^{-14}.
$$
For the finite element spaces the defect is nonzero but small compared with
the discrete dissipation. For Bell elements we obtain
$$
  \max_n |R_h^n|=9.2\cdot 10^{-4},\qquad
  \max_n \frac{|R_h^n|}{D_h^n}=7.1\cdot 10^{-3},\qquad
  \frac{\sum_n \tau |R_h^n|}{\sum_n \tau D_h^n}=3.8\cdot 10^{-3}.
$$
For Argyris elements we obtain
$$
  \max_n |R_h^n|=4.4\cdot 10^{-4},\qquad
  \max_n \frac{|R_h^n|}{D_h^n}=3.4\cdot 10^{-3},\qquad
  \frac{\sum_n \tau |R_h^n|}{\sum_n \tau D_h^n}=1.9\cdot 10^{-3}.
$$
The direct defect computation and the balance-based computation agree up to about \(6.5\cdot 10^{-11}\) in both finite element runs. 
Thus the measured defect is not a postprocessing artefact; it is small but clearly above the roundoff-level Fourier defect.

\begin{figure}[htp!]
\centering
\includegraphics[width=.99\textwidth,page=3]{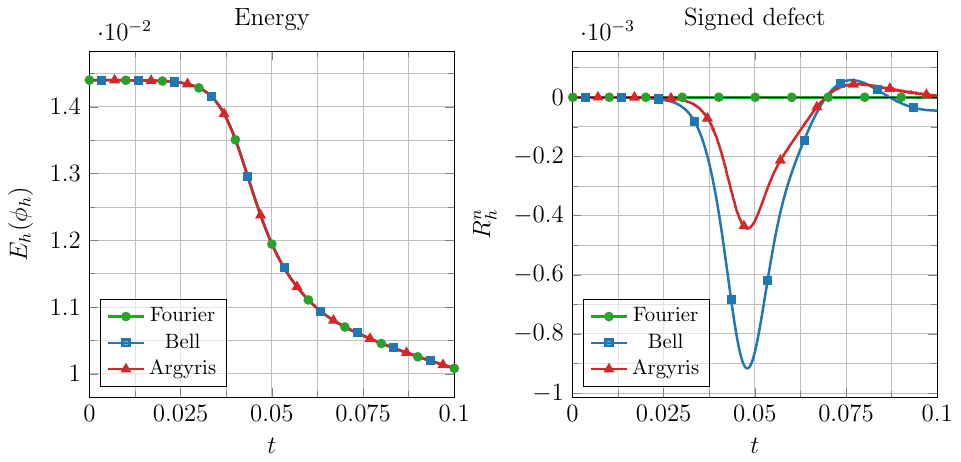}
\caption{Energy evolution and signed energy defect for the Fourier, Bell,
and Argyris discretizations. The energy curves are nearly indistinguishable
on this scale, while the defect plot shows the difference between the
Laplacian-invariant Fourier discretization and the classical \(C^1\) finite
element spaces. Sparse markers are used to distinguish the curves in
black-and-white print.}
\label{fig:energy-and-signed-defect}
\end{figure}

Figure~\ref{fig:spectral-defect} isolates the Fourier defect on a logarithmic
scale. As predicted by the Laplacian-invariance argument, the defect remains
at roundoff level throughout the simulation. This confirms that the nonzero
defects observed for Bell and Argyris elements are not caused by the time
discretization or by the diagnostic itself, but by the lack of exact
Laplacian invariance of the finite element spaces.

\begin{figure}[htp!]
\centering
\includegraphics[width=.6\textwidth,page=4]{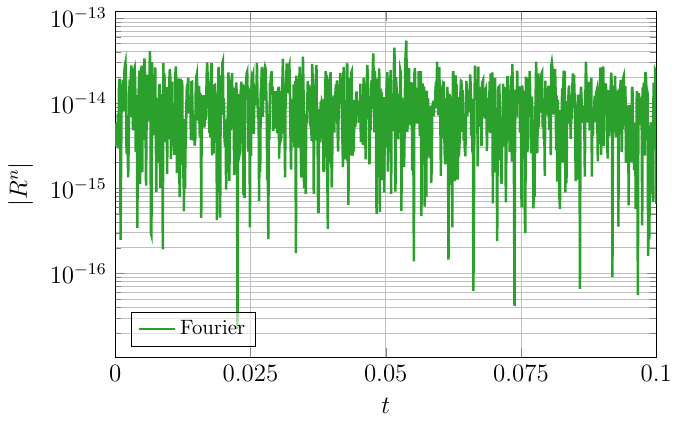}
\caption{Absolute defect for the Fourier discretization. The defect remains
at roundoff level throughout the simulation, confirming the vanishing-defect
property for Laplacian-invariant spaces.}
\label{fig:spectral-defect}
\end{figure}

Finally, Figure~\ref{fig:defect-ratio} compares the finite element defects
with the corresponding discrete dissipation. The left panel shows that
\(|R_h^n|\) remains several orders of magnitude below \(D_h^n\) for both
Bell and Argyris elements. The right panel gives the pointwise and cumulative
relative defects. In particular, the cumulative ratios stay below
\(4\cdot10^{-3}\), showing that the total defect contribution is small
relative to the total Cahn--Hilliard dissipation over the simulated time
interval.

\begin{figure}[htp!]
\centering
\includegraphics[width=.42\textwidth,page=5]{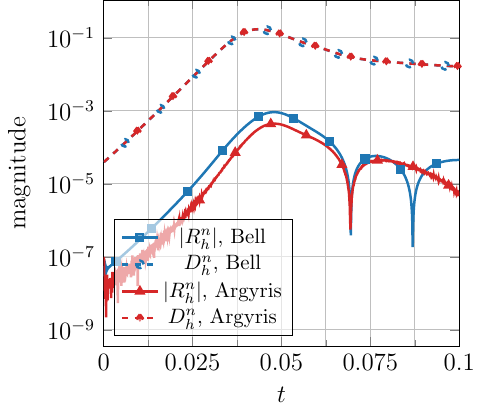}\includegraphics[width=.42\textwidth,page=6]{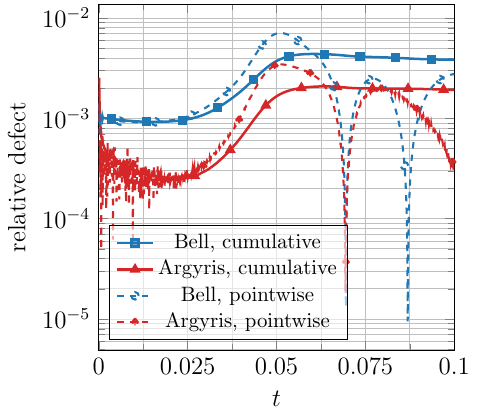}
\caption{Left: Comparison of the absolute direct defect \(|R_h^n|\) and the
discrete dissipation
\(D_h^n=\|\delta_\tau\phi_h^{n+1}\|_{H^{-1}_h}^2\) for Bell and Argyris
elements. Right: Pointwise relative defect \(|R_h^n|/D_h^n\) and cumulative
relative defect
\(\sum_n\tau |R_h^n|/\sum_n\tau D_h^n\). }
\label{fig:defect-ratio}
\end{figure}

\appendix
\section{Discrete elliptic reconstruction}\label{app:reconstruction}

In this appendix, we collect the discrete elliptic reconstruction estimate
used in \Cref{lem:Drel-estimate} and, through it, in the augmented
relative-energy estimate of \Cref{thm:aug-stability}. The result states
that the inverse discrete Laplacian, viewed as a map from
$\mathring V_h\subset L^2(\Omega)$ into $V_h$, satisfies an $H^2$ stability
estimate analogous to the continuous elliptic regularity bound. The proof
proceeds in two steps: first, the inverse discrete Laplacian is identified
with the Poisson Ritz projection of the continuous inverse Laplacian; then
the Ritz projection is shown to be $H^2$-stable using the quasi-interpolant
of \Cref{lem:interpolant} together with the inverse inequality
\eqref{eq:inverse-H2-H1}. Both steps are essentially standard, but the
combination is what allows the relative defect $\mathcal D_h$ to be
controlled in the discrete $H^{-1}$ metric inherent to the Cahn--Hilliard
gradient-flow structure.

We define the mean-free Poisson Ritz projection
$P_h:\mathring H^1_{\per}(\Omega)\to\mathring V_h$ by
\begin{equation}\label{eq:poisson-ritz}
(\nabla(P_hv-v),\nabla\chi_h)=0
\qquad\forall \chi_h\in\mathring V_h.
\end{equation}

\begin{lemma}[$H^2$ stability of the Poisson Ritz projection]\label{lem:poisson-ritz-H2}
Assume \Cref{lem:interpolant} and the inverse inequality
\eqref{eq:inverse-H2-H1}. Then
\begin{equation}\label{eq:poisson-ritz-H2}
\|P_hv\|_{H^2}\le C\|v\|_{H^2}
\qquad\forall v\in\mathring H^2_{\per}(\Omega).
\end{equation}
\end{lemma}

\begin{proof}
By the triangle inequality,
$$
\|P_hv\|_{H^2}
\le
\|\mathcal I_hv\|_{H^2}
+\|P_hv-\mathcal I_hv\|_{H^2}.
$$
Using \eqref{eq:I-H2-stab} and \eqref{eq:inverse-H2-H1},
$$
\|P_hv\|_{H^2}
\le
C\|v\|_{H^2}
+Ch^{-1}\|P_hv-\mathcal I_hv\|_{H^1}.
$$
By the triangle inequality and C\'ea's lemma for \eqref{eq:poisson-ritz},
$$
\|P_hv-\mathcal I_hv\|_{H^1}
\le
\|P_hv-v\|_{H^1}
+\|v-\mathcal I_hv\|_{H^1}
\le
C\inf_{\chi_h\in\mathring V_h}\|v-\chi_h\|_{H^1}
+\|v-\mathcal I_hv\|_{H^1}.
$$
Choosing the mean-corrected interpolant
$$
\chi_h=\mathcal I_hv-\frac{(\mathcal I_hv,1)}{|\Omega|}
\in\mathring V_h
$$
and using \eqref{eq:I-H1-H2}, we obtain
$$
\|P_hv-\mathcal I_hv\|_{H^1}\le Ch\|v\|_{H^2}.
$$
This proves \eqref{eq:poisson-ritz-H2}.
\end{proof}

\begin{proposition}[Discrete elliptic reconstruction]\label{prop:reconstruction}
For all $z_h\in\mathring V_h$,
\begin{equation}\label{eq:reconstruction}
\left\|\Delta\bigl((-\Delta_h)^{-1}z_h\bigr)\right\|
\le C\|z_h\|.
\end{equation}
\end{proposition}

\begin{proof}
Let $p\in\mathring H^1_{\per}(\Omega)\cap H^2_{\per}(\Omega)$ solve
$$
(\nabla p,\nabla v)=(z_h,v)
\qquad\forall v\in\mathring H^1_{\per}(\Omega).
$$
By periodic elliptic regularity,
$$
\|p\|_{H^2}\le C\|z_h\|.
$$
Let $w_h:=(-\Delta_h)^{-1}z_h$. Comparing the definition of $w_h$ with the
Poisson problem above shows that $w_h=P_hp$. Hence
$$
\|\Delta w_h\|
\le
\|w_h\|_{H^2}
=
\|P_hp\|_{H^2}
\le
C\|p\|_{H^2}
\le
C\|z_h\|.
$$
\end{proof}

{\small 
\bibliographystyle{abbrv}
\bibliography{references}
}

\end{document}